\documentclass{amsart}

\usepackage{amsmath, enumerate, amsfonts, amssymb, amsthm, graphics, graphicx, verbatim, caption, subcaption}

\usepackage[english]{babel}
\usepackage[margin=0.95in]{geometry}
\usepackage{tikz, float}
\usetikzlibrary{shapes}
\tikzstyle{subgroup}=[scale=1]

\newtheorem{thm}{Theorem}[section]
\newtheorem{prop}[thm]{Proposition}
\newtheorem{cor}[thm]{Corollary}
\newtheorem{lem}[thm]{Lemma}

\theoremstyle{definition}

\newtheorem{cons}[thm]{Construction}

\def\Z#1{\textrm{Z}(#1)}
\def\CD#1{\mathcal{CD}(#1)}
\def\cdm#1{m^*(#1)}

\def\ord#1{\vert #1 \vert}
\newcommand{\iso}{\cong}
\def\wt#1{\widetilde{#1}}
\def\MM#1{\mathcal{M}_{#1}}
\newcommand{\LA}{\left\langle}
\newcommand{\RA}{\right\rangle}

\title{\bf Chermak-Delgado Lattice Extension Theorems}

\author{Lijian An}
\address{Lijian An, Shanxi Normal University, Department of Mathematics \\ Linfen, China 041004}
\email{sxlf\_alj@163.com}

\author{Joseph Brennan}
\address{Joseph Brennan, Binghamton University, Department of Mathematical Sciences, Binghamton, New York 13902}
\email{jbrennan@binghamton.edu}

\author{Haipeng Qu}
\address{Haipeng Qu, Shanxi Normal University, Department of Mathematics, Linfen, China 041004}
\email{orcawhale@163.com}

\author{Elizabeth Wilcox}
\address{Elizabeth Wilcox, State University of New York at Oswego, Mathematics Department, Oswego, New York 13126-3599}
\email{elizabeth.wilcox@oswego.edu}

\date{\today}

\begin{document}


\begin{abstract}
If $G$ is a finite group with subgroup $H$ then the {\it Chermak-Delgado measure of $H$ (in $G$)} is defined as $\ord H \ord {C_G(H)}$.  The Chermak-Delgado lattice of $G$, denoted $\CD G$, is the set of all subgroups with maximal Chermak-Delgado measure; this set is a sublattice within the subgroup lattice of $G$.  In this paper we provide an example of a $p$-group $P$, for any prime $p$, where $\CD P$ is lattice isomorphic to $2$ copies of $\MM {4}$ (a quasiantichain of width $2$) that are adjoined maximum-to-minimum.  We introduce terminology to describe this structure, called a $2$-string of $2$-diamonds, and we also give two constructions for generalizing the example.  The first generalization results in a $p$-group with Chermak-Delgado lattice that, for any positive integers $n$ and $l$, is a $2l$-string of $n$-dimensional cubes adjoined maximum-to-minimum and the second generalization gives a construction for a $p$-group with Chermak-Delgado lattice that is a $2l$-string of $\MM {p+3}$ (quasiantichains, each of width $p + 1$) adjoined maximum-to-minimum.
\end{abstract}

\maketitle

The Chermak-Delgado measure was originally defined by A. Chermak and A. Delgado as one in a family of functions from the subgroup lattice of a finite group into the positive integers.  I. Martin Isaacs re-examined one of these function, dubbed it the Chermak-Delgado measure, and proved that subgroups with maximal Chermak-Delgado measure form a sublattice in the subgroup lattice of the group.  B. Brewster and E. Wilcox then demonstrated that, for a direct product, this Chermak-Delgado lattice decomposes as the direct product of the Chermak-Delgado lattices of the factors, giving rise to the attention on the Chermak-Delgado lattice of $p$-groups (for a prime $p$) in this paper and others.

The variety seen in the structure of the Chermak-Delgado lattices of $p$-groups seems inexhaustible; for example, there are many $p$-groups with a Chermak-Delgado lattice that is a single subgroup, a chain of arbitrary length, or a quasiantichain of width $p + 1$.  In this paper we show that for any non-abelian $p$-group $N$ with $N$ is in its own Chermak-Delgado measure and $\Phi (N) \leq \Z N$, there exist two $p$-groups $\mathcal{LE}(m,n)$ and $\mathcal{QE}(n)$ with similar properties and such that the Chermak-Delgado lattices of $\mathcal{LE}(m,n)$ and $\mathcal{QE}(n)$ are the Chermak-Delgado lattice of $N$ with either a $m$-diamond or a quasiantichain of width $p + 1$ (respectively) adjoined at both the maximum and minimum subgroups in the Chermak-Delgado lattice of $N$.

\

Let $G$ be a finite group and $H \leq G$.  The {\it Chermak-Delgado measure of $H$ (in $G$)} is $m_G(H) = \ord H \ord {C_G(H)}$.  When $G$ is clear from context we write simply $m(H)$.  For the maximum Chermak-Delgado measure possible in $G$ we write $m^*(G)$ and we use $\CD G$ to denote the set of all subgroups $H$ with $m(H) = m^*(G)$. The proof that this set is actually a modular sublattice in the lattice of subgroups of $G$, called the {\it Chermak-Delgado lattice of $G$}, can be found in \cite{cd1989} and is also discussed in \cite[Section 1G]{Isaacs}. 

Of particular note regarding $\CD G$ are the properties: If $H, K \in \CD G$ then $\langle H, K \rangle = HK$, $C_G(H) \in \CD G$, and also $C_G(C_G(H)) = H$.  This latter property is typically referred to as the ``duality property'' of the Chermak-Delgado lattice.  It is also known that the maximum subgroup in $\CD G$ is characteristic and the minimum subgroup is characteristic, abelian, and contains $\Z G$.

To describe the lattices constructed wherein we introduce the following terms for a positive integer $n$. A {\it quasiantichain of width $n$} will be denoted by $\mathcal{M}_{n+2}$.  An {\it $n$-diamond} is a lattice with subgroups in the configuration of an $n$-dimensional cube. These structures form the most common {\it components} used in this paper, though ``component'' can refer to a lattice of any configuration. A {\it (uniform) $n$-string} is a lattice with $n$ lattice isomorphic components, adjoined end-to-end so that the maximum of one component is identified with the minimum of the other component.  A {\it mixed $n$-string} is a lattice with $n$ components adjoined in the same fashion, though with at least one component not lattice isomorphic to the remaining components.


%
\begin{figure}[t]
\centering
\begin{subfigure}[b]{0.3\textwidth}
	\centering
	\caption{A $3$-string of $2$-diamonds.}
\begin{tikzpicture}
\node (ZP) at (1,0) [subgroup] {$\bullet$};
\node (AO) at (0,1) [subgroup] {$\bullet$};
\node (AE) at (2,1) [subgroup] {$\bullet$};
\node (A) at (1,2) [subgroup] {$\bullet$};
\node (CAO) at (0,3) [subgroup] {$\bullet$};
\node (CAE) at (2,3) [subgroup] {$\bullet$};
\node (P) at (1,4) [subgroup] {$\bullet$};
\node (M1) at (0,5) [subgroup] {$\bullet$};
\node (M2) at (2,5) [subgroup] {$\bullet$};
\node (M) at (1,6) [subgroup] {$\bullet$};
\draw (ZP) to (AE);
\draw (ZP) to (AO);
\draw (AO) to (A);
\draw (AE) to (A);
\draw (A) to (CAO);
\draw (A) to (CAE);
\draw (CAO) to (P);
\draw (CAE) to (P);
\draw (P) to (M1);
\draw (P) to (M2);
\draw (M1) to (M);
\draw (M2) to (M);
\end{tikzpicture}
\end{subfigure}
\begin{subfigure}[b]{0.3\textwidth}
	\centering
	\caption{A mixed $3$-string with two $\MM {7}$ components and a chain of length $2$.}
\begin{tikzpicture}
\node (ZP) at (2,0) [subgroup] {$\bullet$};
\node (A11) at (0,1) [subgroup] {$\bullet$};
\node (A12) at (1,1) [subgroup] {$\bullet$};
\node (A13) at (2,1) [subgroup] {$\bullet$};
\node (A14) at (3,1) [subgroup] {$\bullet$};
\node (A15) at (4,1) [subgroup] {$\bullet$};
\node (P) at (2,2) [subgroup] {$\bullet$};
\node (D) at (2,3) [subgroup] {$\bullet$};
\node (ZM) at (2,4) [subgroup] {$\bullet$};
\node (M) at (2,6) [subgroup] {$\bullet$};
\node (B11) at (0,5) [subgroup] {$\bullet$};
\node (B12) at (1,5) [subgroup] {$\bullet$};
\node (B13) at (2,5) [subgroup] {$\bullet$};
\node (B14) at (3,5) [subgroup] {$\bullet$};
\node (B15) at (4,5) [subgroup] {$\bullet$};
\draw (ZP) to (A11);
\draw (ZP) to (A12);
\draw (ZP) to (A13);
\draw (ZP) to (A14);
\draw (ZP) to (A15);
\draw (P) to (A11);
\draw (P) to (A12);
\draw (P) to (A13);
\draw (P) to (A14);
\draw (P) to (A15);
\draw (ZM) to (B11);
\draw (ZM) to (B12);
\draw (ZM) to (B13);
\draw (ZM) to (B14);
\draw (ZM) to (B15)
;\draw (M) to (B11);
\draw (M) to (B12);
\draw (M) to (B13);
\draw (M) to (B14);
\draw (M) to (B15);
\draw (ZM) to (D);
\draw (P) to (D);
\end{tikzpicture}
\end{subfigure}
\begin{subfigure}[b]{0.3\textwidth}
\centering
\caption{A $2$-string of $3$-diamonds.}
\begin{tikzpicture}
\node (ZP) at (1,0) [subgroup] {$\bullet$};
\node (A11) at (0,1) [subgroup] {$\bullet$};
\node (A12) at (1,1) [subgroup] {$\bullet$};
\node (A13) at (2,1) [subgroup] {$\bullet$};
\node (A21) at (0,2) [subgroup] {$\bullet$};
\node (A22) at (1,2) [subgroup] {$\bullet$};
\node (A23) at (2,2) [subgroup] {$\bullet$};
\node (P) at (1,3) [subgroup] {$\bullet$};
\draw (ZP) to (A11);
\draw (ZP) to (A12);
\draw (P) to (A23);
\draw (P) to (A21);
\draw (P) to (A22);
\draw (ZP) to (A13);
\draw (A11) to (A21);
\draw (A11) to (A22);
\draw (A12) to (A21);
\draw (A12) to (A23);
\draw (A13) to (A23);
\draw (A13) to (A22);
\node (B11) at (0,4) [subgroup] {$\bullet$};
\node (B12) at (1,4) [subgroup] {$\bullet$};
\node (B13) at (2,4) [subgroup] {$\bullet$};
\node (B21) at (0,5) [subgroup] {$\bullet$};
\node (B22) at (1,5) [subgroup] {$\bullet$};
\node (B23) at (2,5) [subgroup] {$\bullet$};
\node (M) at (1,6) [subgroup] {$\bullet$};
\draw (P) to (B11);
\draw (P) to (B12);
\draw (M) to (B23);
\draw (M) to (B21);
\draw (M) to (B22);
\draw (P) to (B13);
\draw (B11) to (B21);
\draw (B11) to (B22);
\draw (B12) to (B21);
\draw (B12) to (B23);
\draw (B13) to (B23);
\draw (B13) to (B22);
\end{tikzpicture}
\end{subfigure}
\caption{Examples of uniform and mixed strings with quasiantichain, chain, and $m$-diamond components.}
\end{figure}
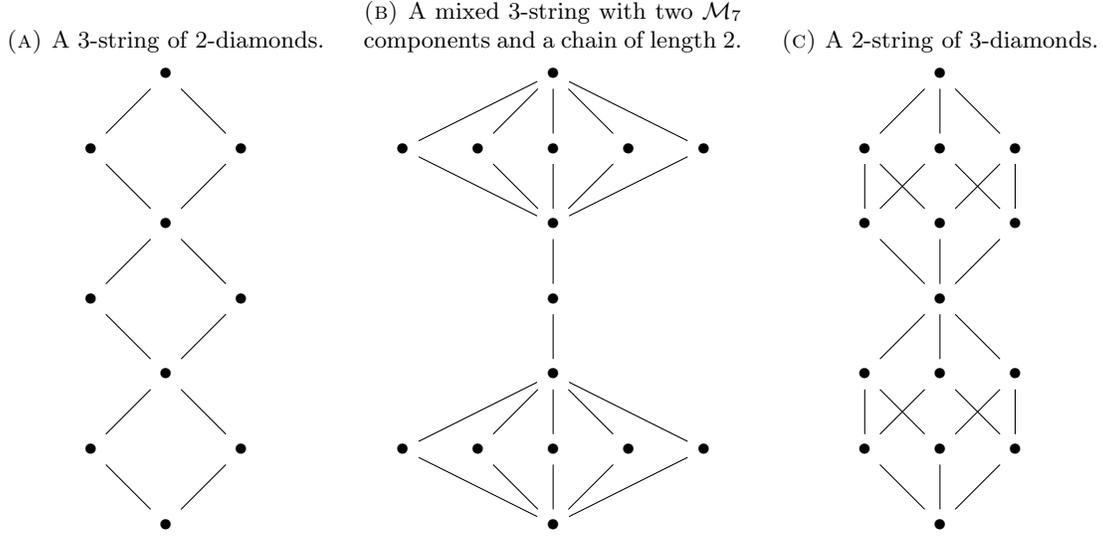

In Section~\ref{DD} we produce an example of a $2$-string of $2$-diamonds, laying the foundation for the proofs in later sections.  In Sections~\ref{nDiamonds} and \ref{qac} we start with a positive integer $n$ and a non-abelian $p$-group $N$ with $N \in \CD N$ and $\Phi(N) \leq \Z N$ and create a group with Chermak-Delgado lattice that is a mixed $n+1$-string with $\CD N$ as the center component.  In Section~\ref{nDiamonds} the remaining $n$ components are $m$-diamonds for a fixed $m \geq 2$ and in Section~\ref{qac} the remaining $n$ components are $\MM {p+3}$.  As a corollary we show there exists a $p$-group $P$ with $P \in \CD P$ and $\Phi (P) \leq \Z P$ such that $\CD P$ is a $2l$-string of $m$-diamonds or $\MM {p+3}$ for all positive integers $l$.  We additionally describe circumstances under which a $2l+1$-string may be constructed.

\section{Example: A $2$-String of $2$-Diamonds}\label{DD}

To give the flavor of the techniques used in later sections, we present the original construction that motivated the main theorems of the paper:  the construction of a $p$-group $P$ with $\CD P$ a $2$-string of $2$-diamonds that contains $P$.  

\begin{thm}\label{example} For any prime $p$ there exists a $p$-group $P$ with $\CD P$ a $2$-string of $2$-diamonds, meaning that $\CD P = \{ \Z P, A_1, A_2, A, AB_1, AB_2, P \}$ where $\Z P < A_i < A < AB_j < P$ for $1 \leq i, j \leq 2$. Moreover, $\Phi(P) \leq \Z P$. \end{thm}

\begin{cons}\label{doublediamond}
For any integer $m > 1$, let $P$ be the group generated by $\{a_i, b_i \mid 1\leq i \leq 2m\}$ subject to the defining relations:
\begin{gather*}
[a_i,a_j]^p = [a_i,b_j]^p = a_i^p = b_i^p = 1 \textrm{ for all $1 \leq i, j\leq 2m$},\\
[a_i,b_j] \neq 1 \textrm{ for $i \not\equiv j$ mod $2$, } [b_i,b_j] \neq 1 \textrm{ for all $1\leq i, j \leq 2m$},\\
\textrm{all other commutators between generators equal 1, and}\\
\textrm{all commutators are in $\Z P$.}\\
\end{gather*}\end{cons}

From the definition it's clear that $\Phi (P) \leq \Z P$ and $\Z P$ is elementary abelian.  Counting the non-trivial commutator relations gives $\ord {\Z P} = p^{4m^2 - m}$ and $\ord P = p^{4m^2 + 3m}$. Define the following subgroups of $P$: 
$$
\begin{array}{lll}
A = \langle a_i \mid 1\leq i \leq 2m \rangle Z(P), && A_1 = \langle a_{2i-1} \mid 1\leq i \leq m \rangle Z(P),\\
A_2 = \langle a_{2i} \mid 1\leq i \leq m \rangle Z(P), && B_1 = \langle b_{2i-1} \mid 1\leq i \leq m \rangle Z(P), \textrm{ and}\\
B_2 = \langle b_{2i} \mid 1\leq i \leq m \rangle Z(P). &&\\
\end{array} 
$$
It's straightforward to verify $C_P(A_1) = AB_1$ and $C_P(A_2) = AB_2$.  Notice, too, that the maximal abelian subgroups have order less than or equal to $|A|$. If we let $z$ be the integer (dependent upon $m$) such that $\ord {\Z P} = p^z$ then observe the following orders:
$$
\ord A = p^{2m + z}, \quad |A_0| = |A_1| = p^{m+z}, \quad \textrm{ and } \quad \ord {AB_1} = \ord {AB_2} = p^{3m+z}.
$$
Therefore the Chermak-Delgado measures of the above groups are all equal, yielding $\cdm P \geq \ord P \ord {\Z P}$. We show this is exactly $m^*(P)$, thereby establishing the theorem.

\

\noindent
{\it Proof of Theorem~\ref{example}.} To prove the theorem we first establish that the minimal subgroup in $\CD P$ is a subgroup of $A$, then determine it must be exactly one of $\Z P$, $A_1$, $A_2$, or $A$.  All of these have the same Chermak-Delgado measure; therefore we establish $m^*(P)$ and use duality to finish the proof.

To begin, let $H \leq P$ be such that $\ord H \leq \ord A$ and $H \in \CD P$.  Suppose $x \in H$ can be written $x=ry$ where $r\in \{b_i \mid 1\leq i \leq 2m\}$ and $y \in \LA \{a_i, b_i\mid 1\leq i \leq 2m\} - \{r\}\RA Z(P)$. The center of $P$ is elementary abelian and the non-trivial commutators of the generators are linearly independent generators of $\Z P$, therefore
$$ \ord{P:C_P(x)} = \ord{x^P} = \ord{[x,P]} \geq \ord{[r,P]} = \ord{r^P} = \ord{P:C_P(r)}. $$
Counting the generators which stabilize $r$ under conjugation gives
$$ |C_P(x)| \leq |C_P(r)| = \frac{|P|}{p^{3m-1}}.$$

That $H \in \CD P$ and $x\in H$ together imply
$$ |H||C_P(x)| \geq |H||C_P(H)| \geq |P||Z(P)|. $$
From $m>1$ and as $\ord H \leq \ord A$ we obtain a contradiction:
$$ p^{2m} = |A:Z(P)| \geq \ord {H : \Z P} \geq |P:C_P(x)| = p^{3m-1};$$
thus $H\leq A$.

Yet the orders of abelian subgroups in $P$ are bounded above by $\ord A$; therefore the minimal subgroup in $\CD P$ must be a subgroup of $A$.  To determine this minimal subgroup, let $H \in \CD P$ be such that $\Z P < H \leq A$.  If $a_1y_1 \in H$ for $y_1 \in \LA a_i\mid 2\leq i \leq 2m\RA Z(P)$ and $a_{2}y_2 \in H$ for $y_2 \in \LA a_1, a_i\mid 3\leq i \leq 2m \RA Z(P)$ then $C_P(H) \leq A$.  Notice that if $z = b_1^{k_1} \cdots b_{2m}^{k_m}a \in C_P(H)$ for integers $k_i$ and $a\in A$ then commutator calculations in a group with nilpotence class 2 give the following implications:
\begin{gather*}
[a_1y_1,z] = 1 \Rightarrow k_{2i} = 0 \mbox{ for } 1\leq i \leq m \quad \textrm{and}\\
[a_2y_2,z] = 1 \Rightarrow k_{2i-1} = 0 \mbox{ for } 1\leq i \leq m.
\end{gather*}
This generalizes for any $a_k$ with $k$ odd and $a_j$ with $j$ even. Since $C_P(H)\leq A$ and $H < A$,
$$ |H||C_P(H)| \leq |A|^2 = |P||Z(P)| \leq m^*(P);$$
equality holds exactly when $H = A$.

We may now conclude that if $ \Z P < H< A$ then $H \leq A_1$ or $H \leq A_2$. If the former then $C_P(H) = AB_1$ and in the latter case $C_P(H) = AB_2$; therefore $m(H) \leq m(A_1) = m^*(P)$, with equality exactly when $H = A_1$ or $H=A_2$.

Since $A$, $A_1$, $A_2$, and $\Z P$ all have the same Chermak-Delgado measure, all of these subgroups are in $\CD P$.  From the duality of the Chermak-Delgado lattice the centralizers $AB_1$, $AB_2$, and $P$ are also in $\CD P$.  Additionally the duality gives that there can be no subgroups $H \in \CD P$ with $A < H < P$ besides those already described, completing the proof of the theorem.\qed

\

It is worth noting that $\ord {A_1} = \ord {A_2}$ is not necessary for achieving a $2$-string of $2$-diamonds.  If we instead let $A_1 = \langle a_i \mid 1 \leq i \leq n - 1 \rangle$ and $A_2 = \langle a_i \mid n \leq i \leq 2n \rangle$, and similarly adjust the commutativity relations so that $C_P(A_1) = \langle b_1 \mid n \leq i \leq 2n \rangle A_1$ and $C_P(A_2) = \langle b_i \mid 1 \leq i \leq n - 1 \rangle A_2$, then the proof still holds.  The result is a Chermak-Delgado lattice that is $2$-string of $2$-diamonds where the subgroups in the diamonds each have distinct order.

\section{$m$-Diamond Lattice Extension Theorem}\label{nDiamonds}

\begin{thm}\label{DiamondExt} Let $N$ be a $p$-group such that $N \in \CD N$ and $\Phi(N) \leq \Z N$.  For any integers $m \geq 1$ and $n \geq 2$ there exists a $p$-group $\mathcal{LE}(m,n)$ and a normal embedding of $N$ into $P$, resulting in $\CD P$ being a mixed $3$-string with center component isomorphic to $\CD N$ and the remaining components being $m$-diamonds. \end{thm}

\begin{cons}\label{GemExtension} Choose $m \geq 1$ and $n \geq 2$.  
\begin{enumerate}
\item For all $i, j$ such that $1 \leq j \leq n$ and $1 \leq i \leq m$, choose distinct $a_{ij}$ with order $p$ and define $A$ to be the direct product of all $\langle a_{ij} \rangle$.


\item Suppose that $N / \Z N = \langle x_1 \rangle \Z N \times \cdots \times \langle x_r \rangle \Z N$ and choose distinct $z_{ijr}$ of order $p$ for $1 \leq i \leq m$, $1 \leq j \leq n$, and $1 \leq k \leq r$.  For all $i, j, t$ such that $1 \leq i, j \leq n$ and $1 \leq t \leq m$, choose distinct $\wt z_{ijt}$ with order $p$. For every $u$ and $v$ with $(1,1) \leq u < v \leq (m,n)$ under the lexicographic ordering, choose a distinct $z_{v}^{u}$ of order $p$. From these generators define:
$$
Z_N = \prod\limits_{\substack{1 \leq i \leq m\\ 1 \leq j \leq n\\ 1 \leq k \leq r}} \langle z_{ijk} \rangle,  \qquad 
Z_A = \prod\limits_{\substack{1 \leq i \leq n\\ 1 \leq j \leq n\\ 1 \leq t \leq m}} \langle \wt z_{ijt} \rangle, \qquad \textrm{ and } \quad  Z_B = \prod\limits_{(1,1) \leq u < v \leq (m,n)} \langle z_{v}^{u} \rangle.
$$
Define $Z = Z_A \times Z_N \times Z_B$ and $\wt N = N \times Z \times A$.  Notice that $\Z {\wt N} = \Z N \times Z \times A$.  Let $\wt A = \Z {\wt N}$.
\item For each $i, j$ such that $1 \leq i \leq m$ and $1 \leq j \leq n$, choose distinct $b_{ij}$ with order $p$.  Define $P = \wt N \rtimes \langle b_{ij} \mid 1 \leq i \leq m, 1 \leq j \leq n \rangle$.  Under this construction the following conjugation relations are observed:\end{enumerate}\end{cons}
\vspace{-10 pt}
\begin{gather*}
[x_k, b_{ij}] = z_{ijk} \textrm{ for all $i, j, k$}, [z, b_{ij}] = 1 \textrm{ for all $z \in \Z N \times Z$},\\
[b_{u},b_{v}] = z_v^u \textrm{ for all $(1,1) \leq u < v \leq (m,n)$},\\
[a_{i'j'},b_{ij}] = 1 \textrm{ for all $i' \ne i$, and } [a_{ti},b_{tj}] = \wt z_{ijt} \textrm{ for all $i$, $j$, $t$}.\\
\end{gather*}

\vspace{-10 pt}
We prove that $P$ defined here exactly fits the requirements for the group $\mathcal{LE}(m,n)$ described in Theorem~\ref{GemExtension}. From the construction of $P$ it's clear that $\Phi (P) \leq \Z P = \Z {\wt N} \times Z$ and $\Phi (P) = \Phi (N) \times Z$, and also that $\Z P$ is elementary abelian.   By counting generators and commutators, one can determine that the exponents on $\ord {P / \Z P}$ and $\ord {\Z P}$ are $2mn + r$ and $\frac{1}{2} mn(2n + 2r + mn - 1) + z$, respectively, where $\ord {\Z N} = p^z$.  Additionally observe $\ord {\wt A} \ord {\wt N} = \ord P \ord {\Z P}$, where $C_P(\wt A) = \wt N$ and $C_P(\wt N) = \wt A$.  This gives that $m(P) = m(\wt N) \leq m^*(P)$.  

To establish the structure of $\CD P$ we define additional subgroups of $\wt A$ and their centralizers in $P$.  For $k$ with $0 \leq k \leq m$, let $\Delta_k$ be a $k$-subset of $\Omega = \{1, 2, \dots, m\}$ and let $A_{\Delta_k} = \langle a_{ij} \mid i \in \Delta_k, 1 \leq j \leq n \rangle$.  Let $\wt {A}_{\Delta_k} = A_{\Delta_k}\Z P$ and $\wt {A}_k = \{ \wt {A}_{\Delta_k} \mid \Delta_k \textrm{ a $k$-subset of } \Omega \}$.  From this definition it is clear that $\wt A_k$ has precisely $\binom{m}{k}$ subgroups. Moreover, any subgroup in $\wt{A}_k$ has a centralizer of the form $\wt{B}_{\Delta_k} = B_{\Delta_k^c} \wt N$ where $B_{\Delta_k^c} = \langle b_{ij} \mid i \not\in \Delta_k, 1 \leq j \leq n \rangle$.

Notice that $m(\wt A_{\Delta}) = m(\wt A)$ for all $\Delta \subseteq \Omega$.  Ultimately we will show that $m^*(P) = m(P)$ and, for every $k$ with $1 \leq k \leq m$, the set $\wt {A}_k \subset \CD P$. This gives the bottom component, an $m$-diamond, in the $3$-string.  The centralizers $C_P(\wt A_{\Delta})$ where $\Delta \subseteq \Omega$ give the second $m$-diamond that forms the top component, part of $\CD P$ under the duality property. Notice $\CD N \iso \CD {\wt N}$, as $\wt N$ is the direct products of $N$ by abelian groups \cite{bw2012}, which will give the center component of $\CD P$.

We begin by examining the subgroups in $\CD P$ with order no greater than $\ord {\wt A}$.

\begin{prop}\label{lemma1} If $H \in \CD P$ and $\ord H \leq \ord {\wt A}$ then there exists some $\Delta \subseteq \Omega$ such that $H = \wt A_{\Delta}$. \end{prop}

\begin{proof} Let $H \in \CD P$ and suppose $\ord H \leq \ord {\wt A}$; then $m^*(P) = m(H) \geq m(\wt A) = \ord {\wt A} \ord {\wt N}$ yields $\ord {C_P(H)} \geq \ord {\wt N}$.  If $C_P(H) = \wt N$ then $H = C_P(\wt N) = \wt A$, using the duality of the Chermak-Delgado lattice.  In the following arguments we assume that $C_P(H) \ne \wt N$ and so there exists $x \in C_P(H) - \wt N$.  

Suppose that $x = wy$ where $w \in \{b_{ij} \mid 1 \leq i \leq m, 1 \leq j \leq n\}$ and $y \in \langle \{b_{ij}\} - \{w\} \rangle \wt N$.  The center of $P$ is elementary abelian and and the non-trivial commutators of generators of $P$ are linearly independent generators of $\Z P$; therefore
$$
\ord {P:C_P(x)} = \ord {x^P} = \ord {[x,P]} \geq \ord {[w,P]}.
$$
Equality holds if and only if $y \in \wt N$; this can be verified straightforwardly using the bilinearity of commutators in $P$.  Thus $C_p(x) \leq C_p(w)$; we calculate $C_P(w)$ in order to place an upper bound on $\ord {C_P(x)}$.  

Fix $u, v$ so that $w = b_{uv}$ and let $\Delta = \Omega - \{u\}$.  The only generators of $P$ that commute with $w$ are precisely $w$ itself and those $a_{ij}$ where $i \ne u$.  Therefore $C_P(w) = \langle w \rangle \wt A_{\Delta}$.  The duality property of the Chermak-Delgado lattice yields: $H = C_P(C_P(H)) \leq C_P(x) \leq \langle w \rangle \wt A_{\Delta}$. Therefore $\ord H \leq p \ord {\wt A_{\Delta}}$.  Recalling that $m(H) \geq \ord {\wt A} \ord {\wt N}$, one may observe that 
$$
\ord {C_P(H)} \geq \frac{\ord {\wt A} \ord {\wt N}}{p \ord {\wt A_{\Delta}}} = p^{n-1} \ord {\wt N}.
$$
If $\ord {C_P(H)} = p \ord {\wt N}$ then $n = 2$.  Thus, since $m(H) \geq \ord {\wt A} \ord {\wt N}$, we can say $\ord H p \geq \ord {\wt A}$.  Yet $H \leq \langle w \rangle A_{\Delta}$ and $\ord {A:A_{\Delta}} = p^2$; therefore $H = \langle w \rangle A_{\Delta}$ and 
$$
C_P(H) = C_P(w) \cap C_P(A_{\Delta}) \leq C_P(x) = H.
$$
Thus $\ord {C_P(H)} \leq \ord {H} \leq \ord {\wt A}$, contradicting the choice of $H$.

Suppose instead that $\ord {C_P(H)} > p \ord {\wt N}$; there exists an $x' = w'y' \in H$ where $w' \in \{ b_{ij} \mid 1 \leq i \leq m, 1 \leq j \leq n\} - \{w\}$ and $y' \in \langle \{b_{ij}\} - \{w'\} \rangle \wt N$. Apply the previous argument to $x'$, arriving at $H \leq \langle x' \rangle \wt A_{\Delta'}$ (where $\wt A_{\Delta'} \in \wt A_{m-1}$, possibly $\Delta = \Delta'$).  This gives $H \leq C_P(w) \cap C_P(w') \leq \wt A_{\Delta}$, which implies $H \leq \wt A$.

Therefore if $H \in \CD P$ and $\ord H \leq \ord {\wt A}$ then $H = \wt A$ or $H \leq \wt A_{\Delta}$ where $\ord {\Delta} = m - 1$. We prove by induction that if $H \in \CD P$ and $H \leq \wt A$ then $H = \wt A_{\Delta}$, where $\Delta$ is now any subset of $\Omega$.  Assume that $\Delta$ is any subset of $\Omega$ such that $H \in \CD P$ but $H < \wt A_{\Delta}$.  

In this case $\ord H \ord {C_P(H)} \geq \ord {\wt A_{\Delta}} \ord {\wt B_{\Delta}}$, since $\wt B_{\Delta} = C_P(\wt A_{\Delta})$.  By order considerations, it's clear that $\ord {C_P(H)} > \ord {\wt B_{\Delta}}$.  Yet $\wt B_{\Delta} < C_P(H)$ because $H \leq \wt A_{\Delta}$, so there must exist an element $x \in C_P(H) - \wt B_{\Delta}$.  Therefore, without loss of generality, there exists $i' \in \Delta$ and $j$ with $1 \leq j \leq n$ such that $x = b_{i'j}y$ for some $y \in \langle b_{ij} \mid i \ne i', 1 \leq j \leq n \rangle \wt N$.  Notice that the bilinearity of the commutator in $P$ gives
$$
1 = [h, b_{i'j}y] = [h, b_{i'j}][h,y] \textrm{ for all } h \in H \implies 1 = [h, b_{i'j}]  \textrm{ for all } h \in H.$$
Thus no element of $H$, written as a product of generators of $\wt A$, contains $a_{i'j}$ as a factor.  Therefore $H \leq \wt A_{\Delta-\{i'\}}$.

By induction, we know that if $H \in \CD P$ and $\ord H \leq \ord {\wt A}$ then $H \leq \wt A_{\Delta}$, for some $\Delta \subseteq \Omega$.  However, if $\ord {\wt A_{\Delta - \{i\}}} < \ord H \leq \ord {A_{\Delta}}$ for any $i \in \Delta$ then $C_P(H) = \wt B_{\Delta}$.  Therefore $m(H) \leq m(\wt A_{\Delta})$ with equality if and only if $H = \wt A_{\Delta}$.  This proves that if $H \in \CD P$ and $\ord H \leq \ord {\wt A}$ then $H = \wt A_{\Delta}$ for some $\Delta \subseteq \Omega$. \end{proof}

In the proof of Theorem~\ref{DiamondExt} we show that there is a subgroup of order less than $\wt A$ in $\CD P$, thereby Proposition~\ref{lemma1} automatically generates an $m$-diamond above the minimal subgroup in $\CD P$.  To prove this, and also to describe the structure of $\CD P$ above the $m$-diamond, we prove:

\begin{lem}\label{lemma2}If $H \in \CD P$ and $\ord H \leq \ord {\wt N}$ then $H \leq \wt N$. \end{lem}

\begin{proof} Let $H \in \CD P$ with $\ord H \leq \ord {\wt N}$, so that $m^*(P) \geq \ord {\wt A} \ord {\wt N}$ implies $\ord {C_P(H)} \geq \ord {\wt A}$.  By way of contradiction, suppose that there exists $x \in H - \wt N$. Then there exists $w \in \{b_{ij} \mid 1 \leq i \leq m, 1 \leq j \leq n\}$ and $y \in \langle \{ b_{ij} \} - \{w\} \rangle \wt N$ such that $x = wy$.  As in the proof of Proposition~\ref{lemma1}, we note that
$$
\ord {P:C_P(x)} =\ord {x^P} = \ord {[x,P]} \geq \ord {[w, P]}
$$
 and therefore $C_P(H) \leq C_P(w)$.  Fix $u, v$ so that $w = b_{uv}$ and let $\Delta = \Omega - \{u\}$.  The only generators that commute with $w$ are precisely $w$ itself and those $a_{ij}$ where $i \ne u$, and therefore $C_P(w) = \langle w \rangle \wt A_{\Delta}$.  Thus 
$$
\ord {C_P(H)} \leq \ord {C_P(x)} \leq \ord {\langle w \rangle \wt A_{\Delta}} = \frac{\ord {\wt A}}{p^{n-1}}.
$$
This contradicts the earlier statement that $\ord {C_P(H)} \geq \ord {\wt A}$; therefore if $H \in \CD P$ and $\ord H \leq \ord {\wt N}$ then $H \leq \wt N$. \end{proof}

These two lemmas are enough to prove that $P$ has the desired Chermak-Delgado lattice.

\

\noindent
{\it Proof of Theorem~\ref{DiamondExt}.} Let $P$ be as described above; we first consider abelian subgroups in $\CD P$.  Let $H\in \CD P$ be abelian and assume, by way of contradiction, that $\ord H > \ord {\wt N}$. There exists an element $b_{i'j}y \in H$ such that $y \in \langle \{b_{ij} \mid 1 \leq i \leq m, 1 \leq j \leq n \} - \{b_{i'j}\} \rangle \wt N$. As $C_P(b_{i'j}) = A_{\Delta}$ for $\Delta = \Omega - \{i'\}$, it follows that $H\not\leq C_P(H)$.  Thus if $H \in \CD P$ is abelian then $H \leq \wt N$. 

We show that if $H \in \CD P$ and $H \leq \wt N$ then $m(H) = m(P)$, and since the minimal member of $\CD P$ is a subgroup of $\wt N$ this is enough to establish that $m^*(P) = m(P)$. Proposition~\ref{lemma1} already shows that if $H \in \CD P$ with $H \leq \wt A$ then $m_P(H) = m_P(P)$, so we consider the case where $\wt A < H < \wt N$.

Let $H \leq P$ with $C_P(H) \not\leq \wt N$; then there exists $b_{i'j}y \in C_P(H)-\wt N$ with $y \in \langle \{ b_{ij} \mid 1 \leq i \leq m, 1 \leq j \leq n \} - \{b_{i'j}\} \rangle \wt N$. It follows that every $h \in H$ must be an element of $\wt A_{\Delta}$ where $\Delta \subseteq \Omega - \{i'\}$.  Therefore if $H \in \CD P$ with $\ord {\wt A} < \ord H < \ord {\wt N}$ then $C_P(H) = C_{\wt N}(H)$ and $m^*(P) = m_P(H) = m_{\wt N}(H)$.  However, $m_P(H) = \ord {\wt A} \ord {\wt N} = m_{\wt N} (\wt N) = m^*(\wt N)$, by designation of $N \in \CD N$ (and hence $\wt N \in \CD {\wt N})$.  This gives $m_P(H) = m_P(\wt N) = m_P(P)$. 

Therefore $m^*(P) = \ord {\wt N} \ord {\wt A}$.  This implies that if $H \in \CD {\wt N}$ then $m_{\wt N} (H) = m_P(H)$ and hence $H \in \CD P$.  Therefore $\{\wt A_{\Delta}, \wt B_{\Delta} \mid \Delta \subseteq \Omega \} \cup \CD {\wt N} \subseteq \CD P$.  In view of Proposition~\ref{lemma1} and Lemma~\ref{lemma2}, we need only show that if $H > \wt N$ and $H \in \CD P$ then $H = C_P(\wt A_{\Delta})$ for some $\Delta \subseteq \Omega$.  However, if $H \in \CD P$ with $H > \wt N$ then order considerations and Proposition~\ref{lemma1} give a $\Delta \subseteq \Omega$ such that $C_P(H) = A_{\Delta}$.  The duality property of the Chermak-Delgado lattice yields $H = C_P(A_{\Delta})$ as required. \qed

\

As $\Phi (P) \leq \Z P$ for the resulting group $P$ we may reiterate the construction of this section $l$ times to produce a group with a Chermak-Delgado lattice that is a $2l+1$-string with center component isomorphic to $\CD N$ and all remaining components $m$-diamonds.  Notice that for any two subgroups $H, K$ in the resulting $m$-diamonds where there does not exist $M$ in the Chermak-Delgado lattice such that $H < M < K$ we have $\ord {K  H} = p^n$.  As a result:

\begin{cor} Let $l, m, n$ be integers with $l, m \geq 1$ and $n \geq 2$. There exists a $p$-group $P$ such that $P \in \CD P$ and $\CD P$ is a $2l$-string of $m$-diamonds.  Moreover, if $H, K \in \CD P$ and there does not exist $M \in \CD P$ with $H < M < K$ then $\ord {K:H} = p^n$. 

If there exists a non-abelian $p$-group $N$ such that $N \in \CD N$ and $\Phi (N) \leq \Z N$ with the appropriate indices between subgroups in $\CD N$ then one can construct a $p$-group $P$ with the same properties such that $\CD P$ is a $2l + 1$-string of $m$-diamonds.  In particular, if there exists a non-abelian $p$-group $J$ with $\Phi(J) \leq \Z J$ and $\CD J= \{J, \Z J\}$then the desired $2l + 1$-string of $m$-diamonds can be constructed. \end{cor}

\begin{proof} For a group $P$ with $\CD P$ being a $2l$-string of $m$-diamonds, let $N = 1$ and reiteratively apply Theorem~\ref{DiamondExt} $l$ times.  The resulting group is exactly as desired.

For a group $P$ with $\CD P$ being a $2l+1$-string of $m$-diamonds, one must start with a group $N$ in order to reiteratively apply Theorem~\ref{DiamondExt}. Suppose that there exists a non-abelian $p$-group $J$ such that $\Phi(J) \leq \Z J$ and $\CD J = \{J, \Z J\}$, with $\ord {J:\Z J} = p^n$.  Then $N = J \times J \times \cdots \times J$ ($m$ factors) has an $m$-diamond as its Chermak-Delgado lattice with the desired indices.  Moreover that $\Phi (J) \leq \Z J$ implies $\Phi (N) \leq \Z N$, therefore Theorem~\ref{DiamondExt} may be reiterated $l$ times to give the desired group. \end{proof}

Such a group $J$ was constructed in \cite[Proposition 3.3]{bhw2013} for $n = 3$.

\section{Quasiantichain Lattice Extension Theorem}\label{qac}

\begin{thm}\label{QACExt} Let $N$ be a $p$-group such that $N \in \CD N$ and $\Phi (N) \leq \Z N$.  For any integer $n \geq 2$ there exists a $p$-group $\mathcal{QE}(n)$ with $\Phi (P) \leq \Z P$ and a normal embedding of $N$ into $P$, resulting in $P \in \CD P$ and $\CD P$ being a mixed $3$-string with center component isomorphic to $\CD N$ and other components being lattice isomorphic to $\MM {p+3}$.\end{thm}

\begin{cons}\label{QACExtension} Choose $n \geq 2$.  
\begin{enumerate}
\item For all $i, j$ such that $1 \leq j \leq n$ and $i \in \{1,2\}$, choose distinct $a_{ij}$ with order $p$ and define $A$ to be the direct product of all $\langle a_{ij} \rangle$. 


\item Suppose $N / \Z N = \langle x_1 \rangle \Z N \times \langle x_2 \rangle \Z N \times \cdots \times \langle x_r \rangle \Z N$ for a positive integer $r$.  For $i, j, k$ with $1 \leq i \leq m$, $1 \leq j \leq n$, and $1 \leq k \leq r$, choose distinct $z_{ijk}$, each of order $p$. For $i, j$ such that $1 \leq i, j \leq n$, choose distinct $z_{ij}$, each with order $p$. For every $u$ and $v$ with $(1,1) \leq u < v \leq (2,n)$ under the lexicographic ordering, choose a distinct $z_{v}^{u}$ of order $p$. Define:
$$
Z_N = \prod\limits_{\substack{1 \leq i \leq 2\\ 1 \leq j \leq n\\ 1 \leq k \leq r}} \langle z_{ijk} \rangle, \quad
Z_A = \prod\limits_{1 \leq i,j \leq n} \langle z_{ij} \rangle, \quad \textrm{ and } \quad
Z_B = \prod\limits_{(1,1) \leq u < v \leq (2,n)} \langle z_{v}^{u} \rangle.
$$
Let $Z = Z_N \times Z_B \times Z_A$ and $\wt N = N \times Z \times A$.  Notice that $\Z {\wt N} = \Z N \times Z \times A$.  Let $\wt A = \Z {\wt N}$.

\item For each $i, j$ such that $1 \leq i \leq 2$ and $1 \leq j \leq n$, choose distinct $b_{ij}$ with order $p$.  Define $P = \wt N \rtimes \langle b_{ij} \mid 1 \leq i \leq 2, 1 \leq j \leq n \rangle$.  Under this construction the following conjugation relations are observed:\end{enumerate}\end{cons}
\vspace{-10 pt}
\begin{gather*}
[x_k, b_{ij}] = z_{ijk} \textrm{ for all $i$, $j$, $k$}, [z, b_{ij}] = 1 \textrm{ for all $z \in \Z N \times Z$},\\
[b_{u},b_{v}] = z_v^u \textrm{ for all $(1,1) \leq u < v \leq (2,n)$},\\
[a_{i'j'},b_{ij}] = 1 \textrm{ for all $i' \ne i$, and } [a_{ti},b_{tj}] = z_{ij} \textrm{ for $t = 1$ or $2$.}\\
\end{gather*}

\vspace{-10 pt}
We show that $P$ satisfies the requirements of Theorem~\ref{QACExtension}. The main difference between this construction and that of Section~\ref{nDiamonds} is in the generators of $Z_A$.  In the latter, $a_{ti}^{b_{tj}}$ resulted in a different central element for each choice of $t$.  In the present construction $a_{12}^{b_{11}} = a_{22}^{b_{21}}$, for example.  The effect of ``glueing'' the commutators together in this manner is reminiscent of construction of a single quasiantichain as given in \cite{bhw2013a}.

The construction clearly dictates that $\Phi (P) \leq \Z P$ and shows $\Z P$ is elementary abelian.  Counting generators and commutators shows that $\ord {\Z P}$ has exponent $n(3n - 2r +1) + z$ where $\ord {\Z N} = p^z$ and $\ord {P / \Z P}$ has exponent $4n + r$.  

It's straightforward to show that $C_P(\wt N) = \wt A$ and vice versa; this gives $m^*(P) \geq m(P) = \ord P \ord \Z P = \ord {\wt N} \ord {\wt A} = m(\wt A)$.  Other subgroups of interest include $A_k = \langle a_{1j}a_{2j}^k \mid 1 \leq j \leq n \rangle$ for $1 \leq k \leq p - 1$ and $A_{p} = \langle a_{2j} \mid 1 \leq j \leq n \rangle$.  Each of these abelian subgroups has index $p^n$ in $\wt A$.  

We show, through a series of lemmas, that $\CD P = \{\Z P, P, A_k, C_P(A_k) \mid 0 \leq k \leq p \} \cup \CD {\wt N}$.  One $\MM {p+3}$ is formed by $\{ \Z P, A_k \wt A \mid 0 \leq k \leq p \}$ and the second by $\{ \wt N, C_P(A_k), P \mid 0 \leq k \leq p\}$.  Notice that $\CD {\wt N} \iso \CD N$ because ${\wt N}$ is the direct product of $N$ with abelian groups \cite{bw2012}.  

We begin by examining $C_P(A_k)$ for $0 \leq k \leq p$.

\begin{lem} Let $A_k$ be as described.  The centralizer of $A_0$ is $\langle b_{2j} \mid 1 \leq j \leq n \rangle \wt N$ and $C_P(A_p) = \langle b_{1j} \mid 1 \leq j \leq n \rangle \wt N$.  For $k$ with $1 \leq k \leq p - 1$, the centralizer of $A_k$ is $C_P(A_k) = \langle b_{1j}^kb_{2j}^{-1} \mid 1 \leq j \leq n \rangle \wt N$.\end{lem}

\begin{proof} The centralizer of $A_0$ and $A_p$ follow immediately from the conjugation relations given in the construction of $P$.  The structure of $C_P(A_k)$ for $1 \leq k \leq p - 1$ is less obvious; first observe the following: 
$$
[a_{1j}a_{2j}^k, b_{1j}^kb_{2j}^{-1}] = [a_{1j},b_{1j}]^k [a_{2j},b_{2j}]^{-k} = z_{jj}^k z_{jj}^{-k} = 1
$$
for $j$ such that $1 \leq j \leq n$.  Thus $\langle b_{1j}^kb_{2j}^{-1} \mid 1 \leq j \leq n \rangle \wt N \leq C_P(A_k)$.

Let $x \in C_P(A_k)$.  Then $x = b_{11}^{\alpha_{11}}b_{12}^{\alpha_{12}} \cdots b_{2n}^{\alpha_{2n}} y$ where $0 \leq \alpha_{ij} \leq p - 1$ for $1 \leq i \leq 2$, $1, \leq j \leq n$ and $y \in \wt N$. Given the linearity of the commutator, we may assume without loss of generality that $y = 1$ and consider the following commutator:
$$
[a_{1j'}a_{2j'}^k, b_{11}^{\alpha_{11}}b_{12}^{\alpha_{12}} \cdots b_{2n}^{\alpha_{2n}}] = \prod\limits_{1 \leq j \leq n} [a_{1j'},b_{1j}]^{\alpha_{1j}} [a_{2j'},b_{2j}]^{k \alpha_{2j}} = \prod\limits_{1 \leq j \leq n} z_{j'j}^{\alpha_{1j}} z_{j'j}^{k\alpha_{2j}}.
$$
This commutator equals 1 if and only if $\alpha_{1j} + k \alpha_{2j} = 0$ for all $j$.  Therefore there exists $n$ linear equations, each of 2 variables and solution space $\langle \big[ \begin{array}{c}
        k\\
        -1\\
        \end{array} \big] \rangle$, as desired. Thus $C_P(A_k) = \langle b_{1j}^kb_{2j}^{-1} \mid 1 \leq j \leq n \rangle$. \end{proof}

This tells us that $m(P) = m(\wt A) = m(A_k)$ for all $k$ with $0 \leq k \leq p$.  Identifying the centralizers of the subgroups $A_k$ is the first step in showing that if $\wt A \in \CD P$ then there is a component with maximum $\wt A$ that is lattice isomorphic to $\MM {p+3}$ and a second mirrored in the structure above $\wt N$.  The next step is to establish the quasiantichain structure; in a manner similar to that of Section~\ref{nDiamonds} we study subgroups in $\CD P$ that have order less than $\ord {\wt A}$.

\begin{prop}\label{qaclem1} If $H \in \CD P$ and $\ord H < \ord {\wt A}$ then $H \in \{ \wt A, \Z P, A_k \mid 0 \leq k \leq p\}$.  \end{prop}

\begin{proof} Let $H \in \CD P$ and $\ord H \leq \ord {\wt A}$.  Then $\ord {C_P(H)} \geq \ord {\wt N}$ because $m^*(P) = \ord H \ord {C_P(H)} \geq \ord {\wt A} \ord {\wt N}$.   If $\ord {C_P(H)} = \ord {\wt N}$ then $H = \wt A$.  

Suppose, instead, that $\ord {C_P(H)} > \ord {\wt N}$; there exists $x \in C_P(H) - \wt N$ with $x = b_{11}^{\alpha_{11}}b_{12}^{\alpha_{12}} \cdots b_{2n}^{\alpha_{wn}} y$ for $y \in \wt N$. By counting the non-central generators of $\wt A$ that do not commute with the $b_{ij}$ in $x$, notice:
$$
\ord {\wt A: C_{\wt A}(x)} = \ord {x^{\wt A}} = \ord {[x, \wt A]} \geq p^n.
$$
Thus $\ord {C_{\wt A}(x)} \leq p^{1-n} \ord {\wt A}$.  Additionally, $C_{\wt A}(x) = C_{\wt N}(x)$ and the $b_{ij}$ do not commute with one another.  This gives $C_P(x) = \langle x \rangle C_{\wt A} (x)$.  However, $\ord x \leq p^2$ and $\ord { \langle x \rangle \cap C_{\wt A} (x)} \leq p$; therefore $\ord {C_P(x)} \leq p^{1-n} \ord {\wt A}$.

By choice of $H$, we know $m^*(P) = \ord {\wt H} \ord {\wt C_P(H)} \geq \ord {\wt A} \ord {\wt N}$.  The duality property of the Chermak-Delgado lattice gives $H = C_P(C_P(H)) \leq C_P(x)$.  These two facts together give 
$$\ord {C_P(H)} \geq p^{n-1} \ord {\wt N}.$$
If $\ord {C_P(H)} > p \ord {\wt N}$ then there exists an $x' \in C_P(H) - \langle x \rangle \wt N$.  Note that $[x, x'] = 1$ and therefore $H \leq C_P(x) \cap C_P(x') \leq \wt A$, as desired.

Suppose instead that $\ord {C_P(H)} = p \ord {\wt N}$.  From $\ord {C_P(H)} \geq p^{1-n} \ord {\wt N}$ we know $n = 2$.  However $\ord H \leq p \ord {C_{\wt A} (x)}$  and $C_{\wt A} (x) < \wt A$ gives the implication:
$$
\frac{\ord H}{p} \leq \ord {C_{\wt A} (x)} < \ord {\wt A} \implies \ord H \leq p^2 \ord {\wt A}.
$$
Then $\ord H \ord {C_P(H)} < \ord {\wt A} \ord {\wt N}$, contradicting the choice of $H$.

Thus we have shown if $H \in \CD P$ with $\ord H \leq \ord {\wt A}$ then $H \leq \wt A$.  Notice that $m(\wt A) = m(\Z P)$ so we consider $\Z P < H < \wt A$, showing that such $H \in \CD P$ must also have the same measure as $\wt A$.  Let $x \in H$ and write $x = a_{11}^{\alpha_{11}} a_{12}^{\alpha_{12}} \cdots a_{2n}^{\alpha_{2n}} y$ for $y \in \Z P$ and $0 \leq \alpha_{ij} \leq p-1$.  Since $x \in \wt A$ we know that $\wt N \leq C_P(x)$; suppose $z \not\in \wt N$ centralizes $x$.  Let $z = b_{11}^{\beta_{11}} \cdots b_{2n}^{\beta_{2n}}w$ where $w \in \wt N$ and $0 \leq \beta_{ij} \leq p-1$.  The bilinearity of the commutator allows for the computation of $[x,z]$, resulting in 
$$
\prod\limits_{1 \leq i \leq n} (z_{i1}^{\beta_{11}}z_{i2}^{\beta_{12}} \cdots z_{in}^{\beta_{1n}})^{\alpha_{1i}} (z_{i1}^{\beta_{21}}z_{i2}^{\beta_{12}} \cdots z_{in}^{\beta_{2n}})^{\alpha_{2i}}.
$$
If $[x,z] = 1$ then the exponents for each fixed $i$ give a linear equation 
$$
\alpha_{1i} (\beta_{11} + \beta_{12} + \cdots + \beta_{1n}) + \alpha_{2i}(\beta_{21} + \beta_{22} + \cdots + \beta_{2n}) = 0.
$$
Each of the $i$ linear equations must be solved simultaneously in order to produce $z \in C_P(x)$ with $z \not\in \wt N$.  This corresponds to a $2 \times n$ consistent matrix with a unique solution, requiring $\alpha_{1i}$ be a negative scalar multiple of $\alpha_{2i}$. Hence there exists $k \in \{ 0, 1, \dots, p \}$ such that $x \in A_k$.  Moreover, any other $x' \in H$ must fit this same form in order to commute with $z$ and hence $H \leq A_k$ and $C_P(H) = C_P(A_k)$.  If $H < A_k$ then $m(H) < m(A_k)$; because $H \in \CD P$ we know then that $H = A_k$.  

As desired, if $H \in \CD P$ with $\ord H \leq \ord {\wt A}$ then $H \in \{ \Z P, \wt A, A_k \mid 0 \leq k \leq p \}$. \end{proof}

Therefore if there is a subgroup of $\wt A$ in $\CD P$ then there is a component that is lattice isomorphic to $\MM {p+3}$ directly above $\Z P$ in $\CD P$ as well as a second quasiantichain in $\CD P$ between $\wt N$ and $P$.  We turn now to the structure of $\CD P$ between $\wt A$ and $\wt N$.

\begin{lem}\label{qaclem2} If $H \in \CD P$ and $\ord H \leq \ord {\wt N}$ then $H \leq \wt N$.  If $H \in \CD P$ and $\wt A \leq H \leq \wt N$ then $m_P(H) = m_{\wt N}(H)$.  \end{lem}

\begin{proof} Let $H \in \CD P$ and $\ord H \leq \ord {\wt N}$.  Suppose, by way of contradiction, that there exists $x \in H - \wt N$.  Similar consideration as in the proof of Lemma~\ref{qaclem1} gives
$$
\ord {C_P(H)} \leq \ord {C_{\wt A}(x)} \leq p^{1-n} \ord {\wt A}.
$$
Then $m^*(P) = \ord H \ord {C_P(H)} < \ord {\wt A} \ord {\wt N}$.  This contradiction implies that $H \leq \wt N$. 

Now suppose that $\wt A \leq H \leq \wt N$ and $H \in \CD P$.  The generators of $\wt N$ and the elements $\{b_{ij}\}$ do not commute. Combined with the bilinearity of commutators in $P$, this implies no non-trivial element of $H$ can commute with an element from $P - \wt N$.  Thus $C_P(H) = C_{\wt N}(H)$, giving $m^*(P) = \ord H \ord {C_P(H)} = m_{\wt N}(H)$. \end{proof}

We are now prepared to prove that $P$ satisfies the description from Theorem~\ref{QACExt}.

\

\noindent
{\it Proof of Theorem~\ref{QACExt}.}  Let $P$ be described as above.  We first determine $m^*(P)$ by showing all abelian subgroups in $\CD P$ have measure equal to $m(P)$.  Suppose $H \geq \Z P$ and $H \in \CD P$.  Assume, by way of contradiction, that $H \not\leq \wt N$; there exists $x = b_{ij}^{\alpha} y \in H$ where $1 \leq \alpha \leq p$ and $y \in \wt N$.  If $i = 1$ then $C_P(b_{1j}) = A_p$ and if $i = 2$ then $C_P(b_{2j}) = A_0$.  Since $C_P(H) \leq C_P(b_{ij})$, using the bilinearity of the commutator, we know that $C_P(H) < \wt A$.  This directly contradicts the assumption that $H$ is abelian. Therefore $H \leq \wt N$.  If $\wt A < H < \wt N$ then $m_P(H) = m_{\wt N}(H)$.  This forces $m_P(H) = m_P(\wt N) = m_P(P)$.  If $H \leq \wt A$ then Proposition~\ref{qaclem1} immediately gives $m(H) = m(P)$.  

Because the minimum subgroup of $\CD P$ is abelian and contains $\Z P$, we know that $m^*(P) = m(P)$ and that minimum is $\Z P$.  This additionally implies that $\{ \Z P, P, A_k, C_P(A_k) \mid 0 \leq k \leq p\} \cup \CD {\wt N}$ is a subset of $\CD P$.  Proposition~\ref{qaclem1} and Lemma~\ref{qaclem2} give that no subgroups of $\wt N$ other than those listed can be members of $\CD P$.  To finish, suppose there exists $H \in \CD P$ such that $\wt N < H < P$.  The centralizer calculation of the preceeding paragraph holds and shows that $C_P(H) < \wt A$.  The duality property of the Chermak-Delgado lattice then allows for applying Proposition~\ref{qaclem1} to $C_P(H)$ to see that $C_P(H) = A_k$ or $C_P(H) = \Z P$.  If the former is the case then $H = C_P(C_P(H)) = C_P(A_k)$ and if the latter is true then $H = P$.  Therefore $\CD P$ is as described by the statement of the theorem. \qed

\begin{cor}\label{string} Let $l, n$ be integers with $l \geq 1$ and $n \geq 2$. There exists a $p$-group $P$ such that $P \in \CD P$ and $\CD P$ is a $2l$-string with components that are lattice isomorphic to $\MM {p+3}$, for any positive integer $l$.  Moreover, if $H, K \in \CD P$ such that there does not exist $M \in \CD P$ with $H < M < K$ then $\ord {K:H} = p^n$. 

If there exists a non-abelian $p$-group $N$ with $\Phi (N) \leq \Z N$ and $N \in \CD N$ such that $\CD N$ is $\MM {p+3}$ with indices $p^n$ then there exists a non-abelian $p$-group $P$ with $P \in \CD P$ and $\CD P$ being a $2l+1$-string with components that are lattice isomorphic to $\MM {p+3}$ and having the appropriate indices. \end{cor}

The proof of Corollary~\ref{string} is a matter of reiteratively applying Theorem~\ref{QACExt}.  For an even length string let $N = 1$.  The matter of finding a $p$-group $N$ with $N \in \CD N$ and $\CD N$ being a quasiantichain with the needed indices is still open.  Such an $N$ exists when $n = 3$, as described in \cite{bhw2013a}.

\subsection*{Acknowledgements}

The authors would like to thank Ben Brewster of Binghamton University for posing the original challenge to find a group $P$ with $\CD P$ a $2$-string of $2$-diamonds.  The second and fourth authors would like to thank Qinhai Zhang and Shanxi Normal University for the gracious invitation and support during their visit.  

This work was supported by the National Natural Science Foundation of China (grant number 11071150) and the Natural Science Foundation of Shanxi Province (grant numbers 2012011001-3 and 2013011001-1).

\bibliographystyle{amsplain}
\bibliography{references}

\end{document}